\documentclass[a4paper,12pt]{article}
\usepackage{amsmath,amsthm,amsfonts,amssymb,amscd} 
\usepackage{tikz}
\usepackage{color}
\usepackage[active]{srcltx}

\newtheorem{theorem}{Theorem}[section]

\newtheorem{lemma}{Lemma}[section]

\newtheorem{corollary}{Corollary}[section]

\theoremstyle{plain}

\title{Lyapunov exponents and regularity of invariant foliations for partially hyperbolic diffeomophisms on $\mathbb T^3$.}
\author{R. Var\~ao} 

\date{}

\begin{document}
\maketitle

\begin{abstract}

We briefly survey some of the recent results concerning the metric behavior of the invariant foliations for a partially hyperbolic on a three-dimensional manifold and propose a conjecture to characterize atomic behavior for conservative partially hyperbolic homotopic to Anosov (DA) on $\mathbb T^3$. On the other hand we prove that if one of the invariant foliations (stable, center or unstable) of a conservative DA on $\mathbb T^3$ is $C^1$ and transversely absolutely continuous with bounded Jacobians the Lyapunov exponent on this direction is defined everywhere and constant. If the center foliation is this foliaion then the DA diffeomophism is smoothly conjugated to a linear Anosov, in particular Anosov. Another consequence of the main theorem is that it does not exist a conservative Ma\~n\'e's example.
\end{abstract}

\section{Introduction}


A diffeomophism $f:M \rightarrow M$ on a compact manifold is called partially hyperbolic if the tangent bundle of $M$ admits a $Df$-invariant splitting $TM=E^s \oplus E^c \oplus E^u$ such that for any three  $x, y, z \in M$
$$
 \|D_xf v^s \| < \|D_yf v^c\| < \| D_zf v^u\|
$$ 
where $v^s, v^c$ and $v^u$ are unit vectors belonging respectively to $E_x^s, E^c_y$ and $E^u_z.$ An Anosov diffeomorphism is one for which it is possible to split the tangent space solely in an stable and unstable directions.

The subbundles $E^s$ and $E^u$ are the stable and unstable direction. These two subbundles integrate into invariant foliations, known as the stable $\mathcal F^s$ and unstable foliation $\mathcal F^u$ (see \cite{bonatti.diaz.viana} and references therein). The existence of a center foliation (that is an invariant foliation tangent to the $E^c$ direction) does not always exist. But for many cases it does exist, in particular all known transitive partially hyperbolic diffeomophism on a three-dimensional manifold $M^3$ admits a center foliation (\cite{bonatti.wilkinson} for a discussion). One of such diffeomophisms are the Derived from Anosov (DA) on $\mathbb T^3$. We say that $f: \mathbb T^n \rightarrow \mathbb T^n$ is a DA diffeomophism if $f$ is partially hyperbolic and homotopic to a linear Anosov $A:\mathbb T^n \rightarrow \mathbb T^n$, that is $A \in SL(n,\mathbb Z)$ and $A$ is a hyperbolic matrix. $A$ is called the linearization of $f$ and $f$ is semi-conjugated to $A$ (\cite{sambarino.hiperbolicidad.estabilidad} for some properties of the semi-conjugacy).

It is long known that the stable and unstable foliations are absolutely continuous \cite{bonatti.diaz.viana}. That is, given a set $A \subset M$ of full volume, then for volume almost every point the leaf $\mathcal F(x)$ of the point $x \in M$  intersects the set $A$ in full leaf volume. It turns out that for the center foliation such a behavior is not the expected one. Ruelle and Wilkinson \cite{RW.singular} gave a robust example of a volume preserving partially hyperbolic diffeomophism on $\mathbb T^3$ with compact center leaves such that there exist a set of full volume $A$ for which for almost every point $x$ the center leaf $\mathcal F^c(x)$ intersects $A$ in finite points.

Such behavior also occurs for $DA$ diffeomomorphism.
The first example were given by \cite{PTV}, which also gives an open condition for the atomic behavior. All these results as well as the study on the disintegration by \cite{varao} lead us to the following conjecture:

\vspace{0,3cm}
\textbf{Conjecture:} Let $f:\mathbb T^3 \rightarrow \mathbb T^3$ be a volume preserving DA diffeomophisms, then the center foliation has atomic disintegration if and only if $\lambda_A^c \lambda^c_f < 0$, where $\lambda_f$ is the center Lyapunov exponent defined almost everywhere and $\lambda^c_A$ is the center Lyapunov exponent for $A$.
\vspace{0,2cm}

It is only known three types of transitive partially hyperbolic diffeomophisms on a three dimensional manifold, $DA$ is one of those three types, for the other two the center foliation has in fact a much more understandable behavior, see \cite{AVW, varao}.

Now that we have seen that the center foliation often is (in the volume preserving case) non-absolutely continuous, one might expect that the absolute continuity of the center foliation should impose a rigid condition on the diffeomophism. Avila, Viana, Wilkinson \cite{AVW} have shown that for the perturbations of the time one of the geodesic flow on a closed negative curved curvature surface, if the center foliation is absolutely continuous, then the center foliation is a $C^1$ foliation and the perturbation is in fact the time one of some Anosov flow. The continuation of their work, that is on ``part II" similar results are obtained for large classes of partially hyperbolic with compact center leaves (skew-product). The natural analogous of such a result for DA diffeomorphisms should be that absolute continuity, implies smooth conjugacy with its linearizaton. But even if the center foliation is $C^1$ it fails to be smoothtly conjugate to its linearization \cite{varao}. Although the $C^1$ condition is not enough for the rigidity, if you also add the tranversely absolute continuity of the center foliation you get smooth conjugacy (see Corollary \ref{cor:rigidity}).

We say that the foliation $\mathcal F$ is \textit{transversely absolutely continuous with bounded Jacobians} if:  given an angle $\theta \in(0,\pi/2]$,
there exists $K\geq 1$ such that, for any two
transversals $\tau_1, \tau_2$ to
$ \mathcal{F}$ of angle at least $\theta$ with the leaves of $\mathcal F$, and any
measurable set $A \subset \tau_1$ then
\begin{eqnarray*}
 K^{-1} m_{\tau_1}(A) \leq m_{\tau_2}(h^{\mathcal F}_{1,2}(A))\leq K m_{\tau_1}(A),
\end{eqnarray*}
where $m_{\tau_i}$ is the Lebesgue measure on the transversal $\tau_i$ and $h^{\mathcal F}_{1,2}$ is the $\mathcal F$-holonomy from $\tau_1$ to $\tau_2$.

We then analyse which implication a smoothness plus transversely continuous foliation implies for the dynamics:

\begin{theorem}\label{theo:main}
Let $f:\mathbb T^3 \rightarrow \mathbb T^3$ be a DA, conservative diffeomorphism, such that there exists $* \in \{ s, c, u \}$ for which $\mathcal F_f^*$ is $C^1$ and transversely absolutely continuous foliation, then the Lyapunov exponent in the $*$-direction is defined for every point and it is equal to $\lambda_A^*$, where $\lambda_A^*$ is the Lyapunov exponent of the Linearization of $A$.
\end{theorem}

From the proof we can say exactly the value of the Lyapunov exponent
\begin{corollary}
 $\lambda^*_f = \lambda^*_A$.
\end{corollary}

Where the linearization is seen as a partially hyperbolic.

\begin{corollary}
 It is not possible to construct a volume preserving Ma\~ne's example. 
\end{corollary}
\begin{proof}
 The Ma\~n\'e's example (\cite{bonatti.diaz.viana} and references therein) is a deformations of a linear Anosov with three real distinct eigenvalues which is not an Anosov diffeomophism, robustly transitive and preserves the center foliation of the Linear Anosov. If it were volume preserving since $\mathcal F^c_f=\mathcal F^c_A$ it is $C^1$ and transversely absolutely continuous. Hence it should have the center Lyapunov exponent defined everywhere and constant, but there are two periodic points with distinct Lyapunov exponents. Absurd.
\end{proof}

The next corollary has already been proved on \cite{varao}, but we state and prove it here to illustrate the type of application of Theorem \ref{theo:main}. It is also important for the subsequent discussion.

\begin{corollary}\label{cor:rigidity}
 Let $f$ be as above, if the center foliation is $C^1$ and transversely absolutely continuous, then $f$ is $C^1$ conjugated to a linear Anosov, in particular $f$ is Anosov.
\end{corollary}
\begin{proof}
 The center Lyapunov exponent is contant and defined everywhere, hence $f$ is an Anosov diffeomophism. By \cite{gogolev-non.abs.cont} the stable or untable Lyapunov exponent is contant on periodic points, since $f$ is volume preserving all the Lyapunov exponent on each direction is constant on periodic points, by \cite{varao} this implies that $f$ is $C^1$ conjugated to its linearization.
\end{proof}

The rigidity result from Corollary \ref{cor:rigidity} does not extend directly into higher dimension:

\vspace{0.3cm}
\textbf{Example:} Let $A: \mathbb T^4 \rightarrow \mathbb T^4$, $A \in A \in SL(4,\mathbb Z)$ a hyperbolic matrix with eigenvalues $\lambda_1 < \lambda_2 \leq \lambda_3 < \lambda_4$, with $\lambda_1 < 1$ and $\lambda_4 > 1$ these are the stable and unstable directions respectively and the sum of the eigenspaces of the eigenvalues from $\lambda_2$ and $\lambda_3$ form the center direction. Let us consider a perturbation $A\circ h$ of $A$, where $h:\mathbb T^4 \rightarrow \mathbb T^4$ is a volume preserving diffeomophism defined as follows, $h$ is the identity map outside some small ball of a fixed point $ p \in \mathbb T^4$, $h$ is volume preserving and preserves the $E^c_A$ direction, then by Baraviera-Bonnati \cite{baraviera.bonatti} $\int \lambda^c_{A\circ h} < \int \lambda^c_A$, hence $A$ and $A\circ h$ are not $C^1$ conjugate, but $\mathcal F^{c}_{A\circ h}$ is a $C^1$ foliation transversely absolutely continuous with bounded Jacobians, because $\mathcal F^{c}_{A\circ h} = \mathcal F^{c}_{A}$
\vspace{0.3cm}

Although the rigidity result of Corollary \ref{cor:rigidity} is not valid as stated for higher dimension, Theorem \ref{theo:main} shed some light on what type of rigidity result to look for. One should as well use the conjugacy result for higher dimensional Anosov diffeomophism from \cite{Gog.Kali.Sadov.Llave}.

\vspace{0.3cm}
\textbf{Question:} Suppose $f: \mathbb T^4 \rightarrow \mathbb T^4$ is a conservative DA diffeomorphism with two dimensional center foliation, if the center foliation as well as an invariant subfoliation of the center foliation are $C^1$ and transversely absolutely continuous with bounded Jacobian, do we get rigidity (smooth conjugacy with a linear Anosov)?

\section{Proof of Theorem \ref{theo:main}}

From now on we work on the universal cover $\mathbb R^3$, we abuse notation and still use $f$ for the lift, as well as $\mathcal F^*$ for the lifted invariant foliation. Also we suppose the linearization has two unstable directions, therefore the center leaf for the linearization is expanding.

Volume will be denoted as $m$ and $m_{\mathcal F(x)}$ will be the Lebesgue measure on the leaf $\mathcal F_f^*(x)$. For convenience we shall suppose $* = c$, that is $\mathcal F_f^* = \mathcal F_f^c$, and whenever there is a difference in the argument we point it out and work the case.

Consider $B$ to be the invariant topological plane $\mathcal W^{su}_f(0)$ and $0 \in \mathbb R^3$ is a fixed point for $f$, where $\mathcal W^{su}_f(0)$ is the stable leaf of $0$ foliated by unstable leaves. We see $B$ as the ``base", by \cite{andy.hammerlindl-thesis.article} all center leaves of $f$ intersect $B$ at a unique point. It also happens for the other cases, for instance $\mathcal W^{sc}(0)$ intersects every unstable leaf in one point. As well as $\mathcal W^{uc}(0)$ intersects every stable leaf in one point \cite{andy.hammerlindl-thesis.article}.

Now, consider $B_0$ to be the set of $x \in \mathbb R^3$ which distance inside $\mathcal F^c$ to $B$ is greater or equal to $L\in \mathbb R_+$, where $L$ is much bigger than the distance of the semi-conjugacy $H$ to the identity. This gives that the iterations $f^n(B_0)$ gets further and further from $B$. 

Now consider $C_n$ to be the union of the center leaves from the base $B$ to the iterate $B_n=f^n(B_0)$ and $\pi_n : C_n \rightarrow B$, the projection through the center leaf. For the time being we shall suppose that $B$ and $B_n$ are $C^1$ two-dimensional manifolds. 

Given a transversal foliation $\mathcal T$, we denote $\mathcal T(x,\epsilon)$ as the ball inside the leaf $\mathcal T(x)$ centered at $x$ and radius $\epsilon$. Then, since $\mathcal F$ is transversely absolutely continuous with bounded Jacobians each slice $\mathcal T(x,\epsilon)$ has volume on the leaf comparable to the volume of the base $B(x,\epsilon) \subset B$. Therefore: 
 \begin{eqnarray}\label{eq:compare.measure}
m(B^{\mathcal T}(I,\epsilon)) =  m_B (B(x, \epsilon)) [K^{-1},K] m_{\mathcal F(x)}(I)   
 \end{eqnarray}

where $B^{\mathcal T}(I,\epsilon) = \bigcup_{z \in I}\mathcal T(z,\epsilon)$. And we use the notation of $a=b[C^{-1},C]$ for $  C^{-1}b\leq  a\leq C b$.

Another important notation is the following. We call $B$ the base also because each point on $B$ is seen as on height zero. That is, given $x \in \mathbb R^3$, the meaning of $[0,q_n] \subset \mathcal F^c(x)$ is the segment between the point $B \cap \mathcal F^c(x)$ and $B_n \cap \mathcal F^c(x)$, where these two points represent in this case $0$ and $q_n$ respectively. That is $0 = 0(x)$, $q_n=q_n(x)$ but we omit the reference to $x$ as it is always clear which $x$ is being considered.

Let us now define the following measure

\begin{eqnarray}\label{eq:m_x,n}
 m_{x,n}= \frac{\lambda^n}{l(\mathcal F[0,q_n])} m_{\mathcal F(x)},
\end{eqnarray}
where $\lambda$ is the eigenvalue of $A$ associated to the direction $E^c$ and $l(\mathcal F[0,q_n])$ is the length inside $\mathcal F^c(x)$.

\begin{lemma}\label{lemma:conditional.step.k}
There exist $K_0 \in \mathbb R_+$, such that for $I \subset \mathcal F(f^n(x))\cap C_k$ measurable
 $$f^n_* m_{x,k}(I) = \lambda^{-n} [K_0^{-1},K_0] m_{f(x),k+1}(I)$$
\end{lemma}
\begin{proof}
Notice that if $f_x m_{\mathcal F(x)} = m_{\mathcal F(f(x))}$ than we would obtain the following equality $f_* m_{x,k} = \lambda m_{f(x),k+1}$, because these two measures would be a multiple of each other and therefore to check that they are the same we only need to check that they coincide in a given set. But, from their definition we have $ f_*m_{x,k}([0,q_{k+1}])= m_{x,k}([0,q_{k}]) = \lambda^k$ and $\lambda^{-1} m_{f(x),k+1}([0,q_{k+1}]) = \lambda^{-1}\lambda^{k+1}= \lambda^k$.

Although we do not have $f_x m_{\mathcal F(x)} = m_{\mathcal F(f(x))}$, let us see that $f_x m_{\mathcal F(x)} = [K^-1,K]m_{\mathcal F(f(x))}$.
\begin{eqnarray*}
 1 = \frac{m(f(\pi_n^{-1}(B(x,\epsilon)))}{m(\pi_n^{-1}(B(x,\epsilon))} = \frac{m(\pi_{n+1}^{-1}(f(B(x,\epsilon)))}{m(\pi_n^{-1}(B(x,\epsilon))}
\end{eqnarray*}

Let $\mathcal T$ be a given smooth transversal on a neighborhood of $\mathcal F^c(x)$ with angle greater than $\theta$ to $\mathcal F^c(x)$, we may take $\mathcal T$ such that $f(\mathcal T)$ still have angle greater than $\theta$ to $\mathcal F^c(f(x))$.
Since $\mathcal F^c_f$ is $C^1$ and transversely absolutely continuous with bounded Jacobians $Vol(\pi_{n+1}^{-1}(f(B(x,\epsilon)))$ and $Vol(\pi_n^{-1}(B(x,\epsilon))$ is calculated as base times height, since $f$ if a lift of a diffeomomorphism on a compact manifold it has bounded Jacobian, hence the base are comparable. The 
$$1 = [K_0^1,K_0]\frac{f_* m_{x,k}}{f_* m_{f(x),k+1}}.$$

The argument is analogous to obtain the relation for $f^n$.
\end{proof}

\begin{lemma}\label{lemma:growth.rate} There  exist $K_1 \in \mathbb R_+$ such that
 $$ K_1^{-1} \leq \liminf_{n}\frac{\lambda^n}{l(\mathcal F[0,q_n])} \leq \limsup_{n}\frac{\lambda^n}{l(\mathcal F[0,q_n])} \leq K_1$$ 
\end{lemma}
\begin{proof}
 Recall that we suppose the splitting of $A$ to have two unstable direction $T\mathbb T^3 = E^s \oplus E^{wu} \oplus E^{uu}$. Also $H$ sends center leaves to center leaves.
 
 \begin{itemize}

 \item Case $\mathcal F^* = \mathcal F^u$

 We claim that $H(\mathcal F^{u}_f) \subset \mathcal F^{wu + uu}_A$. Let $x, y \in \mathcal F^u_f(y)$, since $A^{-1}H=Hf^{-1}$ we get
 \begin{eqnarray*}
||A^{-n}(H(x)) - A^{-n}(H(y))|| &=& || Hf^{-n}(x) - Hf^{-n}(y)||\\ 
&\leq& ||Hf^{-n}(x) - f^{-n}(x)|| + || f^{-n}(x) - f^{-n}(y)|| \\
& & + \;||f^{-n}(y) - Hf^{-n}(y)||
 \end{eqnarray*}

 The semi-conjugacy is uniformly close to the identity, hence the three terms above are bounded, hence $H(x)$ and $H(y)$ cannot have a stable component. The semi-conjugacy $H:\mathbb R^3 \rightarrow \mathbb R^3$ satisfies $H\circ f=A\circ H$, and is from a distance $C$ to the identity. Let $x \in B$, $y \in B_0 \cap \mathcal F^u_f(x)$ and $z \in \mathcal F^c_f(x)$ for which $H(z) \in \mathcal F^{uu}_A(y)$, then
 
 \begin{eqnarray*}
 1 &=& \frac{||A^nH(x)-A^nH(y)||}{||H\circ f^n(x) - H\circ f^n(y)||}\\
 &\geq&  \frac{||A^nH(z)- A^nH(y)||-||A^nH(x)- A^nH(z)|| }{||H\circ f^n(x) - f^n(x) || + || f^n(x) - f^n(y) || + ||f^n(y) - H\circ f^n(y)||}\\
 &\geq& \frac{||A^nH(z)- A^nH(y)||-||A^nH(x)- A^nH(z)||}{2C + ||f^n(x) - f^n(y)||}\\
 &\geq& \frac{(\lambda^{uu}_A)^n ||H(z) - H(y)|| - (\lambda^{wu}_A)^n|| H(x)-H(z)||}{2C + l(\mathcal F[0,q_n])}.
 \end{eqnarray*}
  
 We then get,
 $$\liminf_{n}\frac{\lambda^n|| H(z) - H(y)||}{l(\mathcal F[0,q_n])} \geq 1.$$
 
 Above we used the quasi-isometry of the invariant foliations \cite{andy.hammerlindl-thesis.article}, for instance to say that $||f^n(x) - f^n(y)||$ grows as in the unstable leaf.
 
 Recall that the invariant foliations of $A$ are straight lines. This implies that $|| H(z) - H(y)||$ is uniformly bounded. That comes from euclidian geometry. Consider the triangle formed by the vertices $H(x),H(y),H(z)$ side between the vertices $H(x), H(y)$ is uniformly bounded and the angles of the triangle are unimorly bounded the side formed by the vertices $H(y),H(z)$ is uniformely bounded. Hence
 
  $$\liminf_{n}\frac{\lambda^n}{l(\mathcal F[0,q_n])} \geq K_1.$$

 Similar for the other inequality.
 
 \item Case $\mathcal F^* = \mathcal F^s$ or $\mathcal F^c$.
 
 Since in this case the semi-conjugacy sends stable leaves and center leaves of $f$ to stable leaves and center leaves of $A$ respectively the calculation is similar to the above, except we do not have to consider the point $z$.
 \end{itemize}

\end{proof}

From the above two lemmas we have that $m_{x,i_n}$ is equal to $m_{\mathcal F(x)}$ up to a multiplication by a scalar, this scalar is uniformly bounded. Hence given $x \in \mathbb R^3$ consider a convergent subsequence $a_{x,i_n}\rightarrow a_x$ such that 
 \begin{eqnarray}\label{eq:m_x,n.convergence}
  \lim_{n \rightarrow \infty} m_{x,i_n} = a_x m_{\mathcal F(x)}.
 \end{eqnarray}

We now want to define these measures $m_x$ for every $x \in \mathbb R^3$, but we still want it to keep the property 
\begin{eqnarray}\label{eq:m_x.relation}
 f^n_* m_{x}(.) = \lambda^{-n} [K^{-1},K] m_{f(x)}(.)
\end{eqnarray} 
as in Lemma \ref{lemma:conditional.step.k}.

Given $x$ we let $a_x$ and $i_n(x)$ as in (\ref{eq:m_x,n.convergence}). We now define  $i_n(f(x)) = i_n +1$ for $f(x)$ and because we have convergence of $m_{x,i_n(x)}$ we get convergence of $m_{f(x),i(f(x))}$ to a measure $m_{f(x)}$. And by Lemma \ref{lemma:conditional.step.k} we have (\ref{eq:m_x.relation}).

 

 We now calculate the Lyapunov exponent. By the above we have 
$$\frac{df_*^n m_x}{dm_{f^n(x)}}(f^n(x)) \in \lambda^{-n}[K_1^{-1},K_1].$$ 
Let us calculate the Radon-Nikodym derivative by another way. Let $I_\delta^n \subset \mathcal F^c_{f^n(x)}$ be a segment of length $\delta$ around $f^n(x)$. Then 
$$ \frac{df_*^n m_x}{dm_{f^n(x)}}(f^n(x))= \lim_{\delta \rightarrow 0} \frac{f_*^n m_x(I^n_\delta)}{m_{f^n(x)}(I^n_\delta)}.$$
And using Lemma \ref{lemma:growth.rate} that states that the densities of $m_x$ are uniformely bounded.
\begin{eqnarray*}
\frac{df_*^n m_x}{dm_{f^n(x)}}(f^n(x)) 
&=& \lim_{\delta \rightarrow 0} \frac{m_x(f^{-n}(I^n_\delta))}{m_{f^n(x)}(I_\delta^n) }
=\lim_{\delta \rightarrow 0}\frac{\int_{f^{-n}(I^n_\delta)}  \rho_x d m_{\mathcal F(x)}}{\int_{I^n_\delta} \rho_{f^n(x)} d m_{\mathcal F(f^n(x))} } \\
&\leq& K_1^2 \lim_{\delta \rightarrow 0}\frac{\int_{f^{-n}(I^n_\delta)}  dm_{\mathcal F(x)}}{\int_{I^n_\delta} dm_{\mathcal F(f^n(x))} }
\leq K_1^2 \lim_{\delta \rightarrow 0}\frac{  m_{\mathcal F(x)}(f^{-n}(I^n_\delta)) }{ m_{\mathcal F(f^n(x))}(I^n_\delta) }\\
&\leq& K_1^2 || Df^{-n}(x)||.
\end{eqnarray*}

Therefore

$$ [K_1^{-1},K_1]\lambda^{-n} = \frac{df_*^n m_x}{dm_{f^n(x)}}(f^n(x))  \leq K_1^2 || Df^{-n}(x)||$$
Now apply "$lim_{n \rightarrow \infty} 1/n \;log$"  above to get

$$\lambda^*(x) \leq log \;\lambda.$$

Similarly $\lambda^*(x) \geq log \; \lambda$.

That is the the Lyapunov exponent in the direction $*$ is defined everywhere and equals to $log \; \lambda$.

\hfill $\Box$

\vspace{0,3cm}
\textit{Acknowledgements.} I am greatful for useful conversations with Prof. Ali Tahzibi, Prof. Marcelo Viana and Prof. Amie Wilkinson. As well as FAPESP for financial support (grants \# 2011/21214-3  and  \# 2012/06553-9).

Department of Mathematics, University of Chicago, USA. \textbf{\&}  

Departamento de Matem\'atica, ICMC-USP S\~{a}o Carlos-SP, Brazil 

\textit{Email address:} \textbf{regisvarao@icmc.usp.br}

\end{document}